\documentclass[11pt]{article}

\usepackage{bm}
\usepackage{natbib}

\usepackage{amsmath, amssymb,amsthm}
\usepackage{url}
\usepackage{dsfont}
\usepackage{authblk}
\usepackage{comment}
\newtheorem{theorem}{Theorem}
\newtheorem{lemma}{Lemma}
\newtheorem{corollary}{Corollary}

\newcommand{\given}[2]{\left(#1\parallel #2\right)}
\newcommand{\ip}[2]{\left\langle #1\,,\,#2\right\rangle} 

\DeclareMathOperator*{\argmax}{arg\,max}
\DeclareMathOperator{\one}{\mathds{1}_m}

\begin{document}

\title{Compatible Weighted Proper Scoring Rules\thanks{This is a pre-copyedited, author-produced PDF of an article accepted for publication in Biometrika following peer review. The definitive publisher-authenticated version (Biometrika 99 (4): 989-994, 2012) is available online at \protect\url{http://biomet.oxfordjournals.org/cgi/content/abstract/ass046?ijkey=CaRYBhLvVa4XvRY&keytype=ref}.} }

\author{Peter G. M. Forbes}
\affil{Department of Statistics, University of Oxford\\ 1 South Parks Road, Oxford OX1 3TG, U.K.}
\date{September 19, 2012}

\maketitle 

\begin{abstract} 
Many proper scoring rules such as the Brier and log scoring rules implicitly reward a probability forecaster relative to a uniform baseline distribution.  Recent work has motivated weighted proper scoring rules, which have an additional baseline parameter.  To date two families of weighted proper scoring rules have been introduced, the weighted power and pseudospherical scoring families.  These families are compatible with the log scoring rule: when the baseline maximizes the log scoring rule over some set of distributions, the baseline also maximizes the weighted power and pseudospherical scoring rules over the same set.  We characterize all weighted proper scoring families and prove a general property: every proper scoring rule is compatible with some weighted scoring family, and every weighted scoring family is compatible with some proper scoring rule.
\end{abstract}

\section{Introduction}
Suppose $Y$ is a random variable taking values in $\{1,\ldots,m\}$.  The valid distributions for $Y$ are 
\begin{equation*}
\mathcal P=\{(p_1,\ldots,p_m)^T : 0\leq p_i\leq 1, \sum_{i=1}^m p_i=1\}\subset \mathbb R^m.
\end{equation*} 
A {scoring rule} $s: \mathcal P\times\mathcal P\rightarrow \mathbb R$ is a function linear in its second argument.  The scoring rule is {proper} if $s(p,r)$ is maximized over $p$ at $p=r$, and strictly proper if this maximum is unique.

Consider a forecaster asked to issue a probabilistic prediction $p$ for $Y$.  She is motivated by a reward of $s(p,r)$ upon observing outcome distribution $r$.  If the forecaster's true belief is $p^*$, her expected score $s(p,p^*)$ is maximized when she predicts $p=p^*$.  Hence proper scoring rules encourage honesty.

Two scoring rules {equivalent} if their rewards are linearly related for all $p,r\in\mathcal P$:
\begin{equation}\label{equivalence}
s_1(p,r) = a\{s_2(p,r) + \ip{b}{r}\}
\end{equation}
where $\ip{\cdot}{\cdot}$ is the standard inner product on $\mathbb R^m$, $a>0$ and $b\in\mathbb R^m$.

The main characterization theorem for proper scoring rules was stated by \citet{mccarthy} and proved by \citet{hendrickson}.
\begin{theorem}
A scoring rule $s$ is proper if and only if the function 
\begin{equation}\label{s2Snew}
S(\lambda p)=\lambda s(p,p) 
\end{equation}
defined on $\mathcal P_{\Lambda} = \{\lambda p : \lambda >0, p\in\mathcal P\}$ is convex and satisfies $S(p)\geq s(p, q)$ for all $p,q\in\mathcal P$.  The scoring rule is strictly proper if and only if $S$ is strictly convex on $\mathcal P$.
\end{theorem} 

The function $S$ is called the {optimal expected score}.  \citet{entropy} showed that the negative optimal expected score can be interpreted as a generalized entropy.

When $S$ is differentiable we have \citep{hendrickson}
\begin{equation}\label{S2snew}
s(p,r) = s(\lambda p, r) = \ip{\nabla_{\lambda p} S(\lambda p)}{r}
\end{equation}
which associates a proper scoring rule with any convex differentiable function $S$.  For the rest of this paper we assume that $S$ is twice differentiable on $\mathcal P_\Lambda$, strictly convex on $\mathcal P$, and achieves its unique minimum in $\mathcal P_+$, the interior of $\mathcal P$.

Equation \eqref{S2snew} extends the domain of $s$ to $\mathcal P_\Lambda\times\mathcal P$ and allows us to differentiate $s$ with respect to its first parameter. Since $s(\lambda p, r)=s(p,r)$ for any $\lambda>0$, we have $\ip{\nabla_p s(p,r)}{\one}=0$ for any $p$ and $r\in\mathcal P$, where $\one\in\mathbb R^m$ has all entries equal to one.

Consider a sequence of observations $y_1,\ldots,y_n$ with empirical  distribution $r\in\mathcal P$.  Let $p(\theta)$ be some model which takes values in $\mathcal P_+$ and is differentiable over some open convex set $\Theta$.  Then any scoring rule defines an {optimal score estimator} \citep{gneiting} via
\begin{equation*}
\tilde\theta(r) = \argmax_{\theta\in\Theta} s\{p(\theta),r\}=\argmax_{\theta\in\Theta} \sum_{i=1}^n s\{p(\theta),y_i\}.
\end{equation*}
From \eqref{equivalence}, all equivalent scoring rules have the same optimal score estimator. The optimal score estimator is {well behaved at $r$} if $\tilde\theta(r)$ exists and is the unique root of $\nabla_\theta s\{p(\theta),r\}$ in $\Theta$. When $s$ is the log scoring rule $s(p,r)=\sum_{i=1}^m r_i\log p_i$, the optimal score estimator becomes the maximum likelihood estimator.

A well behaved optimal score estimate $\tilde\theta(r)$ yields the parameter choice that maximizes the forecaster's expected score under the assumption that the future is similar to the past.  Specifically we suppose that our forecaster issues the prediction $p(\theta)$ for some $\theta\in\Theta$.  If she believes that the next observation's distribution is $r$, then $p\{\tilde\theta(r)\}$ maximizes her expected score.

The optimal score estimator can be generalized so that each $y_i$ follows a different probability distribution, as long as these distributions share a common parameter $\theta\in\Theta$.  Thus the optimal score estimator is applicable to regression models that depend on both $\theta$ and some additional covariates. For the sake of brevity we consider only the basic optimal score estimator here, though all the results hold in the general case.

\section{Results}

We define the {baseline} of a strictly proper scoring rule to be the unique $q\in\mathcal P_+$ that maximizes the generalized entropy $-S(p)$.  For example, the log scoring rule's generalized entropy is the Shannon entropy, which is maximized by the uniform distribution. Proper scoring rules tends to give larger rewards for riskier predictions which vary significantly from the baseline.  Given $q\in\mathcal P_+$ and a strictly proper scoring rule $s(p,r)$, there is an equivalent rule with baseline $q$ given by $s(p,r) - s(q,r)$.

A {weighted scoring family} 
\begin{equation*}
s\given{p,r}{\cdot}=\{s\given{p,r}{q} : q\in\mathcal P_+\}
\end{equation*}
is a family of strictly proper scoring rules where each member $s\given{p,r}{q}$ has baseline $q$.  Two weighted proper scoring rules are equivalent if \eqref{equivalence} is satisfied, where now $a$ and $b$ are functions of $q$.  Different members from the same family need not be equivalent.

Weighted scoring families allow us to tailor our scoring rule to the problem at hand, as motivated in \citet{Jose2009} and \citet{scoreMLE}.  This tailoring is achieved by modifying the baseline.  The baseline is easily interpretable and justifiable in many real world situations.  For instance, weighted scoring families are used in \citet{VictorRichmondR} for a optimal portfolio allocation problem, where the baseline corresponds to the market price.

Let $s(p,r)$ be a proper scoring rule and $s\given{p,r}{\cdot}$ be a weighted scoring family.  We say $s\given{p,r}{\cdot}$ is {compatible with} $s(p,r)$ if for any $q$ and $r\in\mathcal P_+$,
\begin{equation}\label{condfordominated}
\left.\nabla_p s(p,r)\right|_{p=q} = \left.a(q)\nabla_p s\given{p,r}{q}\right|_{p=q}
\end{equation}
for some function $a(q)>0$.  In words, equation \eqref{condfordominated} says that the tangent of a weighted scoring rule at its baseline $q$ is parallel to the compatible scoring rule's tangent at $q$.  By approximating $s\given{p,r}{q}$ with its tangent at $p=q$ and applying \eqref{condfordominated}, we obtain
\begin{equation*}
s\given{p,r}{q} \approx s\given{q,r}{q} + \frac{1}{a(q)}\ip{\left.\nabla_p s(p,r)\right|_{p=q}}{p-q}.
\end{equation*}
The first term corresponds to an equivalence factor $\ip{b(q)}{r}$.  Thus, up to equivalence, every member of the weighted scoring family $s\given{p,r}{\cdot}$ is linearly approximated by the compatible proper scoring rule $s(p,r)$ in the vicinity of its baseline.

\begin{theorem}
Any proper scoring rule is compatible with at least one weighted scoring family.  Conversely, every weighted scoring family is compatible with some proper scoring rule, which is unique up to equivalence.
\end{theorem}
\begin{proof}
Let $s(p,r)$ be a proper scoring rule.  From the definition \eqref{condfordominated}, it is compatible with the weighted scoring family where each member is equivalent to $s(p,r)$:
\begin{equation*}
s\given{p,r}{q} = s(p,r) - s(q,r).
\end{equation*}
Conversely, consider the weighted scoring family $s\given{p,r}{\cdot}$.  From \eqref{S2snew} and \eqref{condfordominated}, a proper scoring rule $s(p,r)$ is compatible with this family if and only if its optimal expected score $S(p)$ satisfies
\begin{equation}\label{Shessian}
\nabla^2_q S(q) = a(q)\left.\nabla^2_p S\given{p}{q}\right|_{p=q}
\end{equation}
for some $a(q)>0$ and all $q\in\mathcal P_+$.  The right hand side is a positive definite matrix since it is the Hessian of the convex function $S\given{p}{q}$.  Thus the $S$ satisfying \eqref{Shessian} is convex and corresponds to a strictly proper scoring rule.  This solution is unique up to equivalence since the solution of a second-order differential equation is unique up to a linear term.
\end{proof}

Having shown that a compatible proper scoring rule always exists, we now provide an alternative characterization for compatibility which has direct applications to optimal score estimation and decision theory.

\begin{lemma}\label{lemma1}
A weighted scoring family $s\given{p,r}{\cdot}$ is compatible with the proper scoring rule $s(p,r)$ if and only if $\left.\nabla_\theta s\{p(\theta),r\} \right|_{\theta=\theta_0}=0$ implies $\left.\nabla_\theta s\{p(\theta),r\parallel p(\theta_0\}\right|_{\theta=\theta_0}=0$
for all differentiable models $p(\theta)$, all $\theta_0\in\Theta$, and all $r\in\mathcal P_+$.
\end{lemma}
\begin{proof}
Choose some model $p(\theta)$ and $r\in\mathcal P_+$.  Suppose that $s\given{p,r}{\cdot}$ is compatible with $s(p,r)$, so that \eqref{condfordominated} holds for all $q\in\mathcal P_+$.  Then \eqref{condfordominated} certainly holds when $q=p(\theta_0)$ for any $\theta_0\in\Theta$.  Left multiplying both sides of \eqref{condfordominated} with the matrix $\left. \nabla_\theta p^T(\theta)\right|_{\theta=\theta_0}$ and using the chain rule,
\begin{equation*}
\left.\nabla_\theta s\{p(\theta),r\} \right|_{\theta=\theta_0}=a\{ p(\theta_0)\}\left.\nabla_\theta s\{p(\theta),r\parallel p(\theta_0)\}\right|_{\theta=\theta_0}.
\end{equation*}
Thus if $\left.\nabla_\theta s\{p(\theta),r\} \right|_{\theta=\theta_0}=0$ then $\left.\nabla_\theta s\{p(\theta),r\parallel p(\theta_0)\}\right|_{\theta=\theta_0}=0$.

Conversely, suppose $\left.\nabla_\theta s\{p(\theta),r\} \right|_{\theta=\theta_0}=0$ implies $\left.\nabla_\theta s\{p(\theta),r\parallel p(\theta_0)\}\right|_{\theta=\theta_0}=0$.  When $q=r$, both sides of \eqref{condfordominated} are being evaluated at their critical points and hence are zero.  We will show \eqref{condfordominated} holds for $q\neq r$ by showing that $v=\left.\nabla_p s\given{p,r}{q}\right|_{p=q}$ is parallel to $w=\left.\nabla_p s(p,r)\right|_{p=q}$.  Using \eqref{S2snew} we can rewrite $v$ as
\begin{equation}\label{veqn}
v=\left.\nabla^2_pS\given{p}{q}\right|_{p=q}r
\end{equation}
 where $\nabla^2_p S\given{p}{q}$ is the positive definite Hessian of $S\given{p}{q}$.  This implies $v\neq 0$ since $r\neq0$.  Furthermore since $v$ is a gradient of $s\given{p,r}{q}$, $\ip{v}{\one}=0$.  The same arguments show that $w\neq 0$ and $\ip{w}{\one}=0$.

Suppose $v$ is not parallel to $w$.  Then we can define the non-zero vector
\begin{equation}\label{beqn}
b=v-\frac{\ip{v}{w}}{\ip{w}{w}}w.
\end{equation}
By construction $\ip{b}{w}=0$.  Consider the model $p(\theta)=q+\theta b$ where $\theta$ takes values on $\Theta$, an open neighbourhood of zero small enough such that $\{p(\theta):\theta\in\Theta\}\subset\mathcal P_+$.  It follows from $\ip{v}{\one}=0$ and $\ip{w }{\one}=0$ that $p(\theta)$ is normalized for all $\theta\in\Theta$.  Thus $p(\theta)$ is a valid distribution for $\theta\in\Theta$ and, by our choice of $w$ and $p(\theta)$, 
  \begin{equation*}
  \left.\nabla_\theta s\{p(\theta),r\}\right|_{\theta=0}=\ip{\left.\nabla_\theta p(\theta)\right|_{\theta=0}}{w}=\ip{b}{w}=0.
  \end{equation*}
Hence by assumption, $\left.\nabla_\theta s\{p(\theta),r\parallel q\}\right|_{\theta=0}=0$.  By definition of $v$ we have
$\left.\nabla_\theta s\{p(\theta),r\parallel q\}\right|_{\theta=0}=\ip{b}{v}$
and thus $\ip{b}{v}=0$.  Substituting this into \eqref{beqn}, $\ip{v}{v}{\ip{w}{w}}=\ip{v}{w}^2$ and the Cauchy--Schwarz inequality implies that $v$ is parallel to $w$: $w=a(r,q)v$.  Using \eqref{veqn}, we rewrite $w=a(r,q)v$ as
\begin{equation*}
\left.\nabla_p^2 S\given{p}{q}\right|_{p=q}r = a(q,r)\left.\nabla_p^2 S(p)\right|_{p=q}r.
\end{equation*}
Since both matrices are positive definite, $a(q,r)>0$.  Since the left hand side is linear in $r$, we see $a=a(q)$, which proves \eqref{condfordominated}.
\end{proof}

Consider a forecaster motivated by a weighted scoring rule with baseline $q$ to issue a prediction $p(\theta)$ for $Y$.  She chooses her prediction based on some decision rule $p\{\breve\theta(r)\}$, where $r$ is the empirical distribution of the previous observations of $Y$.  For instance, $\breve\theta$ could be the optimal score estimator for her weighted scoring rule.  Her risk function is $-s[p\{\breve\theta(p^*)\}, p^*\parallel q]$, which depends on the unknown true distribution $p^*$ of $Y$.  Since $p^*$ is unknown it is approximated with the empirical distribution $r$.

Suppose the baseline is determined by the optimal score estimator of the compatible scoring rule, $q=p\{\tilde\theta(r)\}$.  Then, assuming $\breve\theta$ and $\tilde\theta$ to be well behaved at $r$, Lemma 1 implies that the forecaster's risk function is uniquely minimized when she issues the prediction $q$.  The optimal score estimator of the compatible scoring rule dominates any other estimator $\breve\theta$ for this choice of baseline.

\section{Examples}

Define the {quasi-Bregman} weighted scoring families to be the proper scoring rules with optimal expected scores
\begin{equation}\label{quasisep}
S\given{p}{q} = h\left\{\sum_{i=1}^m f(q_i) g\left(\frac{p_i}{q_i}\right) \right\}-g'(1)\sum_{i=1}^m\frac{p_if(q_i)}{q_i} h'\left\{ g(1)\sum_{j=1}^m f(q_j)\right\},
\end{equation}
where $g'$ denotes the derivative of $g$ with respect to its parameter, and similarly for $h'$.  We require that $f$ is positive, $g$ is twice differentiable and strictly convex, and that $h$ is twice differentiable and strictly increasing.  This defines a weighted scoring family for each choice of $f$, $g$ and $h$.  The expected score $S\given{p}{q}$ is strictly convex since $g$ is strictly convex, $f$ is positive and $h$ is increasing.  Hence the quasi-Bregman weighted scoring families are strictly proper.  The second term of \eqref{quasisep} ensures that $S\given{p}{q}$ has baseline $q$, though removing it achieve a simpler, equivalent rule for optimal score estimation.

The weighted power and pseudospherical scoring families of \citet{VictorRichmondR}, defined by
\begin{align*}
s^{\mathrm{pow}}\given{p,r}{q} &=\frac{1-\sum_{i=1}^m p_i^\beta q_i^{1-\beta}}{\beta} - \frac{1-\sum_{i=1}^m r_i p_i^{\beta-1}q_i^{1-\beta}}{\beta-1},\\
s^{\mathrm{ps}}\given{p,r}{q} &= \frac{1}{\beta-1}\left
\{\frac{\sum_{i=1}^m r_i p_i/q_i}{\left(\sum_{i=1}^m p_i^\beta q_i^{1-\beta}\right)^{1/\beta}}-1\right\}
\end{align*}
for $\beta>1$, are quasi-Bregman weighted scoring families with $f(x)=x$ and 
\begin{align*}
h^{\mathrm{pow}}(x)=\frac{x-1}{\beta(\beta-1)},\quad g^{\mathrm{pow}}(x)=x^\beta,\quad
h^{\mathrm{ps}}(x)=\frac{x^{1/\beta}-1}{\beta(\beta-1)},\quad g^{\mathrm{ps}}(x)=x^\beta.
\end{align*}
\citet{scoreMLE} proved that $\left.\nabla_\theta s\{p(\theta),r\} \right|_{\theta=\theta_0}=0$ implies $$\left.\nabla_\theta s\{p(\theta),r\parallel p(\theta_0\}\right|_{\theta=\theta_0}=0$$ when $s\given{p,r}{r}$ is a power or pseudospherical weighted scoring family and $s(p,r)$ is the log scoring rule.  From Lemma 1, this is equivalent to showing that the power and pseudospherical weighted scoring families are compatible with the log scoring rule.

\begin{corollary}\label{corr1}
The log scoring rule is compatible with any quasi-Bregman weighted scoring family with $f(x)=x$.  This holds for any twice differentiable and strictly convex $g$, and any twice differentiable and strictly increasing $h$.
\end{corollary}
\begin{proof} 
By substituting $f(x)=x$ into \eqref{quasisep} and using \eqref{S2snew}, we obtain
\begin{equation}\label{quasibregman}
s\given{p,r}{q} = h'\left\{\sum_{i=1}^m q_i g\left(\frac{p_i}{q_i}\right) \right\}\sum_{j=1}^m r_j g'\left(\frac{p_j}{q_j}\right).
\end{equation}
The log scoring rule is $s(p,r)=\sum_{i=1}^m r_i\log p_i$.  Substituting \eqref{quasibregman} and the log scoring rule into \eqref{condfordominated} shows that the equality holds with $a=h'\left\{g(1)\right\}g'(1)$.  The functions $h$ and $g$ enter only through their values and first derivatives at 1.
\end{proof}

We define the {Bregman weighted scoring families} as the quasi-Bregman weighted scoring families with $h(x)=x$.  By substituting \eqref{quasisep} into \eqref{S2snew} and using equivalence, the Bregman weighted scoring families take the simple form
\begin{equation}\label{weightedBregman}
s\given{p,r}{q}=\sum_{i=1}^m f(q_i)\left\{g\left(\frac{p_i}{q_i}\right)+g'\left(\frac{p_i}{q_i}\right)\frac{r_i-p_i}{q_i}\right\}.
\end{equation}
We recover the unweighted Bregman scoring rules of \citet{entropy}, i.e.,
\begin{equation}\label{unweightedBregman}
s(p,r)=\sum_{i=1}^m \left\{\tilde g(p_i)+\tilde g'(p_i)(r_i-p_i)\right\},
\end{equation}
by using a flat baseline and rescaling $g$ to $\tilde g(p_i)=f(q_i)g(p_i/q_i)=f(m^{-1})g(mp_i)$.  The unweighted Bregman scoring rules are uniquely specified through the convex function $\tilde g$ alone.
\begin{corollary}
The unweighted Bregman rule specified by $\tilde g$ is compatible with all weighted Bregman families with $f(x)=x^2 \tilde g''(x)$.  This holds for any twice differentiable and strictly convex $g$.
\end{corollary}
\begin{proof}
We use \eqref{condfordominated} with $s\given{p,r}{q}$ given by  \eqref{weightedBregman} with $f(x)=x^2 \tilde g''(x)$ and $s(p,r)$ given by \eqref{unweightedBregman}. 
\end{proof}

We illustrate the use of this corollary via an example.  The unweighted power scoring rule is defined by $\tilde g(x)=x^\beta/\{\beta(\beta-1)\}$ for $\beta>1$.  Using the above corollary with $f(x)=x^\beta$, we see that the unweighted power scoring rule is compatible with all weighted scoring families of the form
\begin{equation}
s\given{p,r}{q}=\sum_{i=1}^m q_i^\beta \left\{g\left(\frac{p_i}{q_i}\right)+g'\left(\frac{p_i}{q_i}\right)\frac{r_i-p_i}{q_i}\right\}
\end{equation}
for any choice of $g$.

\section{Discussion}

As an application of compatible proper scoring rules, consider a portfolio allocation problem similar to \citet{VictorRichmondR}.  There is a market consisting of $m$ assets, and a market maker who sets the prices at $q$.  After one time period asset $Y$ will be worth $1$ unit and the other assets will be worthless.  The investor purchases a portfolio, spending a proportion of his wealth $p_i(\theta)$ on each asset and thus receiving $p_i(\theta)/q_i(\theta)$ units of each asset.  He chooses $\theta$ based on the current prices $q$ and the historical outcome distribution $r$.  Suppose the investor's negative risk is given by a weighted scoring rule $s\{p(\theta),r\parallel q\}$.  The market maker does not know the form of the investor's scoring rule, but he believes it to come from a weighted scoring family compatible with some known proper scoring rule.  The market maker prices the assets using the compatible rule's optimal score estimator, $q=p\{\tilde\theta(r)\}$.  Then the market maker's price coincides with the investor's minimal risk portfolio $p(\theta)$: when the pricing is done by a compatible proper scoring rule, the investor is best served by buying the same number of units of each asset.  \citet{mlemotivation} interprets this minimal risk portfolio from an economic perspective, for the special case where the compatible rule is the log score.

Until now, the only weighted scoring families considered in the literature were the weighted power and pseudospherical scoring rules.  Since both are compatible with the log scoring rule, their optimal score estimators are dominated by the maximum likelihood estimator when the baseline is given by the latter.  \citet{scoreMLE} conjectured the existence of a characterization theorem for all weighted proper scoring families whose optimal score estimators are dominated in this way.  They went on to suggest that this theorem might reveal an unrecognized property of the log scoring rule.

We have found their conjectured characterization theorem: the optimal score estimator of any weighted proper scoring rule is dominated by the compatible proper scoring rule's optimal score estimator when the baseline is set to the compatible proper scoring rule's optimal score estimate.  However, instead of revealing a special property of the log scoring rule, we have shown that every proper scoring rule is compatible with some family of weighted proper scoring rules.

\section*{Acknowledgment}
I thank Steffen Lauritzen, Philip Dawid, Tilmann Gneiting and the referees for their helpful comments.


\end{document}